\documentclass[12pt]{article}
\usepackage{amsmath}
\usepackage{amsfonts}
\usepackage{amssymb}

\usepackage{amsthm} 
\usepackage{amscd}
\theoremstyle{plain} 
\newtheorem{theorem}{Theorem}[section]
\newtheorem{corollary}[theorem]{Corollary}

\newtheorem{lemma}[theorem]{Lemma}

\newtheorem{definition}[theorem]{Definition}

\numberwithin{equation}{section}

\newcommand{\al}{\alpha}
\newcommand{\be}{\beta}

\newcommand{\bz}{\mathbf {Z}}

\newcommand{\ot}{\otimes}

\newcommand{\lra}{\longrightarrow}
\newcommand{\lla}{\longleftarrow}

\begin{document}
\title{A Comparison of Products in \\  Hochschild Cohomology}
\author{Jerry M. Lodder}
\date{}
\maketitle

\noindent
{\bf{Abstract.}} 
In this paper we transport Steenrod's cup-$i$ products, $i \geq 0$,
from the singular cochains on the free loop space ${\rm{Maps}}(S^1, \,
BG)$ to Hochschild's original cochain complex ${\rm{Hom}}_k (k[G]^{\ot
  *}, \, k[G])$ defining Hochschild cohomology.  Here $G$ is a
discrete group, $k$ an arbitrary (commutative) coefficient ring, and $BG$ the
classifying space of $G$.  This induces a natural action of the (mod 2)
Steenrod algebra on the Hochschild cohomology of a group ring.
For cochains supported on $BG$, we prove
that Gerstenhaber's cup product agrees with the simplicial cup product
and Gerstenhaber's pre-Lie product agrees with Steenrod's cup-one
product.  As a consequence, for cocycles $f$ and $g$ supported on
$BG$, the Gerstenhaber bracket $[f, \, g] = 0$ in $HH^*(k[G]; \,
k[G])$.  This is interpreted in terms of the Batalin-Vilkovisky
structure on $HH^*(k[G]; \, k[G])$.



\section{Introduction}
Recall that for a group $G$, the cyclic bar construction, 
$N^{\rm{cy}}_*(G)$ \cite[7.3.10]{cyclic-hom}, is a simplicial set whose
geometric realization is a model for the free loop space
$$  {\rm{Maps}}(S^1, \ BG) := {\cal{L}}BG,  $$
where $BG$ is the classifying space of $G$ and $S^1$ denotes the unit
circle.  The free loop space is of interest in string topology and is
a topic of current research.  
The singular cohomology groups $H^* ( {\cal{L}}BG; \, k)$ can
be computed using the model $N^{\rm{cy}}_*(G)$, namely from the
$b^*$ cochain complex, $n \geq  0$, 
$$  \ldots \, \overset{b^*}{\lra}
{\rm{Hom}}_k(k[G]^{\ot (n+1)}, \, k) \overset{b^*}{\lra} 
{\rm{Hom}}_k(k[G]^{\ot (n+2)}, \, k) \overset{b^*}{\lra} \, \ldots \ .$$
Additionally, the simplicial structure of $N^{\rm{cy}}_*(G)$ allows
the construction of Steenrod's cup-$i$ products \cite{Steenrod}, $i
\geq 0$, on the cochain complex 
$$  {\rm{Hom}}_k ( k[G]^{\ot (*+1)}, \, k),  $$ 
using the $b^*$ coboundary without restriction on the
coefficient ring $k$ (often considered as $\bz$ in this paper).  
The homotopy equivalence 
$$  \lambda : |N^{\rm{cy}}_*(G)| \overset{\simeq}{\longrightarrow} 
{\cal{L}}BG  $$
as formulated by Goodwillie \cite{Goodwillie} and others \cite{BF}
\cite[7.3.11]{cyclic-hom} 
induces a quasi-isomorphism of cochain complexes
$$  \lambda^* : C^*({\cal{L}}BG, \, k) \to
{\rm{Hom}}_k( k[G]^{\ot (*+1)}, \, k)  $$
that preserves the cup-$i$ products, where $C^*$ denotes singular
cochains.  The simplicial cup (zero) product is homotopy commutative at the
cochain level with the  homotopy given by two possible compositions
for cup-one products.  All higher homotopies, cup-$(i-1)$ products,  
are themselves homotopy commutative, with the cup-$i$ product providing 
the homotopy between two possible compositions for the cup-$(i-1)$ products.  

The goal of this paper is to transport the cup-$i$ products to
Hochschild's original cochain complex \cite{Hochschild1944} defining
Hochschild cohomology, i.e.,
$$  \ldots  \overset{\delta}{\lra} 
{\rm{Hom}}_k(k[G]^{\ot n}, \, k[G]) \overset{\delta}{\lra} 
{\rm{Hom}}_k(k[G]^{\ot (n+1)}, \, k[G]) \overset{\delta}{\lra}  \ldots \ ,$$
where the coboundary map $\delta$ involves the product in the ring
$k[G]$.  In this way, the cohomology groups $HH^*(k[G]; \, k[G])$
support two product structures, namely the Gerstenhaber product
\cite{Gerstenhaber} and the simplicial cup (zero) product.  
While the Gerstenhaber product is known to be homotopy commutative on 
Hochschild cochains, the two possible compositions to establish this 
equivalence are themselves not homotopy equivalent, as detected by
the Gerstenhaber (Lie) bracket, which, in general, is nonzero on
Hochschild cohomology.  

In this paper we prove that the two products, 
the simplicial and the Gerstenhaber product, agree as cochains 
when evaluated on a subcomplex representing
$BG$, i.e, constant maps of $S^1$ into $BG$, or strings of length
zero.  Also, Steenrod's cup-one
product agrees with Gerstenhaber's pre-Lie product on this
subcomplex.  As a consequence, we prove that for cocycles $f$ and $g$
supported on $BG$, the Gerstenhaber bracket $[f, \, g]$ is zero in
$HH^*(k[G]; \, k[G])$, since $[f, \, g]$ becomes the coboundary of a
cup-two product involving $f$ and $g$.  Thus, in the Batalin-Vilkovisky algebra
\cite{Tradler} on $HH^*(k[G]; \, k[G])$, we have
$ \Delta ( f \cdot g) = \Delta (f) \cdot g +
(-1)^p f \cdot \Delta (g)$, where $f$ and $g$ are cocycles supported on
$BG$, ${\rm{deg}}(f) = p$.  Also, the cup-$i$ products induce a
natural action of the mod 2 Steenrod algebra on $HH^*( \bz /2 [G]; \,
\bz /2 [G])$ that recovers the action of the Steenrod algebra on the
singular cohomology groups $H^* ({\cal{L}}BG; \, \bz /2)$.  
Finally, two operads are seen to
act on ${\rm{Hom}}_k (k[G]^{\ot *}, \, k[G])$, one is the endomorphism
operad \cite[5.2.12]{LV}, giving rise to the pre-Lie product $f \circ
g$.  The other is the sequence operad \cite{Berger, McClure}, giving
rise to the cup-$i$ products $f \underset{i, \, S}{\cdot} g$, and also
used in the proof of the Deligne conjecture \cite{McClure}.  For
cochains $f$ and $g$ supported on $BG$, we have 
$$  f \circ g = f \underset{1, \, S}{\cdot} g.  $$    

Section Two of the paper contains the definitions of the cochain
complexes ${\rm{Hom}}_k (k[G]^{\ot *}, \, k[G])$, $\delta$, and 
${\rm{Hom}}_k( k[G]^{\ot (* + 1)}, \, k)$, $b^*$, as well as the
(injective) cochain map 
$$  \Phi_* : {\rm{Hom}}_k (k[G]^{\ot *}, \, k[G])  
\longrightarrow {\rm{Hom}}_k( k[G]^{\ot (* + 1)}, \, k).  $$
Definition \eqref{supported} and the discussion there offer the
construction of a cochain supported on $BG$.   
This section also contains the definitions of all products 
necessary for the proofs in the following section.    
Section Three transports the cup-$i$ structure from
${\rm{Hom}}_k( k[G]^{\ot (* + 1)}, \, k)$ to 
${\rm{Hom}}_k (k[G]^{\ot *}, \, k[G])$ via a cochain map
$$  \Psi_* : I_* \to {\rm{Hom}}_k (k[G]^{\ot *}, \, k[G]),  $$
where $I_* = {\rm{Im}}\, (\Phi_*)$.  We have $\Psi_* \circ \Phi_* =
{\bf{1}}$ on ${\rm{Hom}}_k (k[G]^{\ot *}, \, k[G])$.  The various
products are compared in this last section.

\section{Hochschild Cohomology}

Let $A$ be an associative algebra over a ground ring $k$ that is
unital, commutative and associative, such as the integers.  
Recall \cite{Hochschild1944, Hochschild1945} that $HH^*(A; \, A)$, 
the Hochschild cohomology of
$A$ with coefficients in A viewed as a bimodule over itself is the
homology of the cochain complex:
\begin{align*} 
& {\rm{Hom}}_k(k, \, A) \overset{\delta}{\lra} {\rm{Hom}}_k(A, \, A)
\overset{\delta}{\lra} \ldots \\
& \ldots  \overset{\delta}{\lra} 
{\rm{Hom}}_k(A^{\ot n}, \, A) \overset{\delta}{\lra} 
{\rm{Hom}}_k(A^{\ot (n+1)}, \, A) \overset{\delta}{\lra}  \ldots \, ,
\end{align*}
where, for a $k$-linear map $f: A^{\ot n} \to A$, $\delta f : A^{\ot
  (n+1)} \to A$ is given by
\begin{align*}
& (\delta f)(a_1, \, a_2, \, \ldots \, , a_{n+1}) = a_1 f(a_2, \, \ldots
\, , a_{n+1})  \, + \\
& \Big( \sum_{i=1}^n (-1)^i f(a_1, \, a_2, \, \ldots \, , a_i a_{i+1}, \,
\ldots \, , a_n) \Big) + (-1)^{n+1} f(a_1, \, a_2, \, \ldots \, , a_n)
a_{n+1}.  
\end{align*}
For the special case of $n = 0$, $(\delta f)(a_1) = a_1 f(1) - f(1)
a_1$.  

These cohomology groups, $HH^*(A; \, A)$, support a construction as an
{\sc{Ext}}-functor over the ring $A \ot A^{\rm{op}}$ \cite[IX.4]{Cartan}.
Also, applying the {\sc{Tor}}-functor to a certain free $A \ot A^{\rm{op}}$
resolution of the product
$$  m:  A \ot A \to A, \ \ \ m(x \ot y) = xy,  $$
we have the following standard resolution for computing $HH_*(A;
\, A)$, the Hochschild homology of $A$ with coefficients in the
bimodule $A$ \cite[X.4]{Maclane}:  
$$  A \overset{b}{\lla} A^{\ot 2} \overset{b}{\lla} \, \ldots \, 
\overset{b}{\lla} A^{\ot n} \overset{b}{\lla} A^{\ot (n+1)}
\overset{b}{\lla} \, \ldots \, ,  $$
where for $(a_0, \, a_1, \, \ldots \, , a_n) \in A^{\ot (n+1)}$,
\begin{align*}
&  b(a_0, \, a_1, \, \ldots \, , a_n) = \\
& \Big( \sum_{i=0}^{n-1} (-1)^i (a_0, \, \ldots \, , a_i a_{i+1}, \,
\ldots \, , a_n) \Big) + (-1)^n (a_n a_0,, a_1, \, \ldots \, , a_n).
\end{align*}
For $n =1 $, $b(a_0, \, a_1) = a_0 a_1 - a_1 a_0$.  Moreover, when $A$
is unital, $\{ A^{\ot (n+1)} \}_{n \geq 0}$ is a simplicial $k$-module
with face maps
\begin{align} \label{face}
& d_i = b_i : A^{\ot (n+1)} \to A^{\ot n}, \ \ \ i = 0, \ 1, \ 2, \,
\ldots \, , \, n, \\
& b_i (a_0, \, a_1, \, \ldots \, , \, a_n) = (a_0, \, \ldots \, , a_i
a_{i+1} \, , \ldots \, , a_n), \ \ \ 0 \leq i \leq n-1,  \\
& b_n (a_0, \, a_1, \, \ldots \, , \, a_n) = (a_n a_0, \, a_1, \,
\ldots \, , a_n),
\end{align}
and degeneracies $s_i : A^{\ot (n+1)} \to A^{\ot (n+2)}$, $i = 0$, 1,
2, $\ldots \,$, $n$,
\begin{align} \label{degeneracies}
& s_i(a_0, \, a_1, \, \ldots \, , \, a_n) = 
(a_0, \, a_1, \, \ldots \, , \, a_i, \, 1, \, a_{i+1}, \, \ldots \, , \,
a_n), \ \ \ 0 \leq i \leq n.
\end{align}

Let $HH^*_{\cal{K}}(A)$ denote the homology of the cochain complex
${\rm{Hom}}_k(A^{\ot (*+1)}, \, k)$ with respect to the $b^*$
coboundary map, i.e., 
\begin{align*} 
&{\rm{Hom}}_k(A, \, k) \overset{b^*}{\lra} {\rm{Hom}}_k(A^{\ot 2}, \, k)
\overset{b^*}{\lra} \, \ldots \\ 
&\ldots \, \overset{b^*}{\lra}
{\rm{Hom}}_k(A^{\ot n}, \, k) \overset{b^*}{\lra} 
{\rm{Hom}}_k(A^{\ot (n+1)}, \, k) \overset{b^*}{\lra} \, \ldots \, ,
\end{align*}
where for a $k$-linear map $\varphi : A^{\ot n} \to k$, $b^*(\varphi)
: A^{\ot (n+1)} \to k$ is given by
$$  b^*(\varphi) (a_0, \, \ldots \, , \, a_n) = \varphi(b(a_0, \,
\ldots \, , \, a_n)). $$
 
Both $HH^*(A; \,A)$ and $HH^*_{\cal{K}}(A)$ inherit graded product
structures from associative products on the respective cochains.  For
$f \in {\rm{Hom}}_k(A^{\ot p}, \, A)$ and 
$g \in {\rm{Hom}}_k(A^{\ot q}, \, A)$, the Gerstenhaber (cup) product
\cite{Gerstenhaber} 
$$  f \underset{G}{\cdot} g \in {\rm{Hom}}_k(A^{\ot (p+q)}, \, A) $$
is given by 
$$  (f \underset{G}{\cdot} g) (a_1, \, a_2, \, \ldots \, , \, a_{p+q})
=  f(a_1, \, \ldots \, , \, a_p) \cdot g(a_{p+1}, \, \ldots \, , \,
a_{p+q}),  $$
where the product above occurs in the algebra $A$.  Then 
$$  \delta  (f \underset{G}{\cdot} g) = (\delta f)
\underset{G}{\cdot} g + (-1)^p f \underset{G}{\cdot} (\delta g).  $$
Thus, if $f \in HH^p(A; \, A)$ and $g \in HH^q(A; \, A)$, then
$f \underset{G}{\cdot} g \in HH^{p+q}(A; \, A)$.  For $\al \in
{\rm{Hom}}_k(A^{\ot (p+1)}, \, k)$ and $\be \in {\rm{Hom}}_k(A^{\ot
  (q+1)}, \, k)$, the simplicial (cup) product 
$\al \underset{S}{\cdot} \be \in {\rm{Hom}}_k(A^{\ot (p+q+1)}, \, k)$
is given by 
$$ (\al \underset{S}{\cdot} \be)(\sigma) = \al (d_{p+1} \, d_{p+2} \, \ldots
\, d_{p+q} (\sigma)) \cdot \be (d_0 \, d_1 \,  \ldots \,  
d_{p-1}(\sigma)),  $$
where $d_{p+1} \, d_{p+2} \, \ldots \, d_{p+q} (\sigma)$ is
the front $p$-face of $\sigma = (a_0, \, a_1, \, \ldots \, , \,
a_{p+q}) \in A^{\ot (p+q+1)}$ and $d_0 \, d_1 \, \ldots \,
d_{p-1}(\sigma) = d^p_0 (\sigma)$ is the back $q$-face of $\sigma$.
The product above is now in the ground ring $k$.  We have 
$$  b^* (\al \underset{S}{\cdot} \be) = b^*(\al) \underset{S}{\cdot}
\be + (-1)^p \al \underset{S}{\cdot} b^*(\be).  $$ 
For $\al \in HH^p_{\cal{K}}(A)$ and $\be \in HH^q_{\cal{K}}(A)$, it
follows that $\al \underset{S}{\cdot} \be \in HH^{p+q}_{\cal{K}}(A)$.  

Gerstenhaber \cite{Gerstenhaber} has shown that on $HH^*(A; \, A)$,
the product $f \underset{G}{\cdot} g$ is graded commutative by using
the idea of function composition, understood today in terms of the
endomorphism operad ${\rm{Hom}}_k(A^{\ot n}, \, A)$ \cite[5.2.12]{LV}.  
Specifically, for $f \in {\rm{Hom}}_k(A^{\ot p}, \, A)$ and 
$g \in {\rm{Hom}}_k(A^{\ot q}, \, A)$, define $f \underset{(j)}{\circ}
g \in {\rm{Hom}}_k(A^{\ot (p+q-1)}, \, A)$ for $j = 0$,  1,  2, 
$\, \ldots \,$,  $p-1$,  by 
\begin{align*}
& (f \underset{(j)}{\circ} g) (a_1, \, a_2, \, \ldots \, , \,
a_{p+q-1}) = \\
& f(a_1, \, \ldots \, , \, a_j, \, g(a_{j+1}, \, \ldots \,
, \, a_{j+q}), \, a_{j+q+1}, \, \ldots \, , \, a_{p+q-1}).  
\end{align*}
Choosing the sign convention $f \circ g = \sum_{j=0}^{p-1}
(-1)^{(p-1-j)(q-1)} f \underset{(j)}{\circ} g$, we have
\begin{align*}
& \delta(f \circ g) = (\delta f) \circ g + (-1)^{p-1} (f \circ \delta
g) + (-1)^p [ f \underset{G}{\cdot} g - (-1)^{pq} g
\underset{G}{\cdot} f].  
\end{align*} 
If $f$ and $g$ are cocycles, then $f \underset{G}{\cdot} g$ and 
$(-1)^{pq} g \underset{G}{\cdot} f$ differ by a coboundary, so that in
$HH^*(A; \, A)$,
$$  f \underset{G}{\cdot} g = (-1)^{pq} g \underset{G}{\cdot} f . $$
Gerstenhaber calls $f \circ g$ a pre-Lie product, since 
$$  [f, \, g] = f \circ g - (-1)^{(p+1)(q+1)} g \circ f  $$ 
induces a Lie bracket on $HH^*(A; \, A)$.    

From the work of Steenrod \cite{Steenrod}, it follows that the simplicial cup product
is graded commutative on the cohomology of any simplicial complex,
although in 1947 Steenrod was writing  before the formulation of the
modern definition of a (semi)simplicial set.  
For $\al \in {\rm{Hom}}_k(A^{\ot (p+1)}, \, k)$
and  $\be \in {\rm{Hom}}_k(A^{\ot (q+1)}, \, k)$, recall that the
cup-one product
$$  \al \underset{1, \, S}{\cdot} \be \in {\rm{Hom}}_k(A^{\ot (p+q)},
\, k) $$
can be written in terms of the face maps $d_i$ as
\begin{align*}
& (\al \underset{1, \, S}{\cdot} \be)(\sigma)  =  \\
& \sum_{j=0}^{p-1} (-1)^{(p-1-j)(q-1)} \, \al ( (
d_{j+1} \, d_{j+2} \, \, \ldots \, \, d_{j+q-1}) (\sigma) )
\cdot \\
& \hspace{1in} \be( (d_0 \, d_1 \, \, \ldots \, \, d_{j-1} \,
d_{j+q+1} \, d_{j+q+2} \, \, \ldots \, \, d_{p+q-1}) (\sigma) ),
\end{align*}
where $\sigma = (a_0, \, a_1, \, \ldots \, , \, a_{p+q-1}) \in A^{\ot
  p+q}$.  
With the above choice of signs, we have:
\begin{align*}
& b^*( \al \underset{1, \, S}{\cdot} \be) = b^*( \al ) \underset{1, \, S}{\cdot} \be 
+ (-1)^{p-1} \al  \underset{1, \, S}{\cdot} b^*( \be ) 
+(-1)^p [ \al  \underset{S}{\cdot} \be - (-1)^{pq} \be  \underset{S}{\cdot} \al ].
\end{align*}
Again, for cocycles $\al$ and $\be$, $\al  \underset{S}{\cdot} \be$
and $(-1)^{pq} \be  \underset{S}{\cdot} \al$ differ by a coboundary.
A description of the cup-$i$ products, $i \geq 0$, 
$$  \al \underset{i, \, S}{\cdot} \be \in {\rm{Hom}}_k(A^{\ot (p+q+1-i)},
\, k) $$ 
in terms of the face maps $d_j$ can be deduced from \cite{McClure,
  Steenrod}.  We use the following sign convention:
\begin{align*}
& b^*( \al \underset{i, \, S}{\cdot} \be) = b^*( \al ) \underset{i, \, S}{\cdot} \be 
+ (-1)^{p-1} \al  \underset{i, \, S}{\cdot} b^*( \be ) \\
& \hspace{.9in} + (-1)^p [ (-1)^{(i-1)(p+q+1)}\al  \underset{i-1, \, S}{\cdot} \be - 
(-1)^{pq} \be  \underset{i-1, \, S}{\cdot} \al ].
\end{align*}
Today the cup-$i$ products are understood in terms of the sequence 
operad \cite{Berger, McClure}.

Now, the group ring $k[G]$ is an algebra over the cyclic operad
\cite[13.14.6]{LV}, meaning that $k[G]$ supports a symmetric, bilinear
inner product
$$  \langle \ , \ \rangle : k[G] \times k[G] \to k  $$ 
satisfying $\langle ab, \, c \rangle = \langle a, \, bc \rangle$, for
all $a$, $b$, $c \in k[G]$.  By definition, for $g$, $h \in G$, 
$$  \langle g, \, h \rangle = \begin{cases} 1, & h = g^{-1} \\
                                            0, & h \neq g^{-1}.
                              \end{cases}  $$
Then $\langle \ , \ \rangle$ is extended to be linear in each
variable, resulting in a $k$-linear map on the tensor product:
$\langle \ , \ \rangle : k[G] \ot k[G] \to k$.   Since 
$\langle \ , \ \rangle$ is symmetric, we also have $\langle a, \, bc
\rangle = \langle ca, \, b \rangle$, i.e.,  $\langle \ , \ \rangle$ is
invariant under a cyclic shift of the product.  Note that with this
structure, $k[G]$ is also a Frobenius algebra.  Although stated for
group rings, the following lemma applies to any (symmetric) Frobenius
algebra.    

\begin{lemma} \label{Phi-lemma}  There is cochain map
$$  \Phi_n : {\rm{Hom}}_k (k[G]^{\ot n} , \, \, k[G]) \to {\rm{Hom}}_k 
 (k[G]^{\ot (n+1)} , \, \, k), \ \ \   n \geq 0, $$
given by 
$$  \Phi_n (f) (g_0, \, g_1, \, g_2, \, \ldots \, , \, g_n) =
\langle g_0, \ f(g_1, \, g_2, \, \ldots \, , \, g_n) \rangle ,  $$
where $f : k[G]^{\ot n} \to k[G]$ is a $k$-linear map and each $g_i
\in G$.  
\end{lemma}
\begin{proof}
For $f \in {\rm{Hom}}_k (k[G]^{\ot (n-1)} , \, \, k[G])$, 
\begin{align*}
&  \Phi_n (\delta f) (g_0, \, g_1, \, \ldots \, , \, g_n)  = \langle
g_0, \ (\delta f)((g_1, \, g_2, \, \ldots \, , \, g_n) \rangle  \\
&  = \langle g_0, \ g_1 f(g_2, \, \ldots \, , \, g_n) \rangle 
+ \sum_{i=1}^{n-1}(-1)^i \langle g_0, \ f(g_1, \, \ldots \, , \, g_i
g_{i+1}, \, \ldots \, , \, g_n) \rangle  \\
& + (-1)^n \langle g_0, \ f(g_1, \, g_2, \, \ldots \, , \, g_{n-1}) g_n
\rangle .
\end{align*}
On the other hand,
\begin{align*}
& b^* ( \Phi_{n-1}(f))(g_0, \, g_1, \, \ldots \, , \, g_n ) =
\Phi_{n-1}(f)( b(g_0, \, g_1, \, \ldots \, , \, g_n )) \\
& = \langle g_0 g_1, \ f(g_1, \, \ldots \, , \, g_n) \rangle 
+ \sum_{i=1}^{n-1} (-1)^i \langle g_0, \ f(g_1, \, \ldots \, , \, g_i
g_{i+1}, \, \ldots \, , \, g_n) \rangle  \\
& + (-1)^{n} \langle g_n g_0, \ f(g_1, \, g_2, \, \ldots \, , \,
g_{n_1}) \rangle.
\end{align*} 
Using the cyclic symmetries of the inner product $\langle \ , \
\rangle$,  we have
$$   \Phi_n (\delta f) =  b^* ( \Phi_{n-1}(f)), \ \ \ n \geq 1 .  $$
\end{proof}

\begin{lemma}  The cochain map 
$$  \Phi_n : {\rm{Hom}}_k (k[G]^{\ot n} , \, \, k[G]) \to {\rm{Hom}}_k 
 (k[G]^{\ot (n+1)} , \, \, k), \ \ \   n \geq 0 , $$
is injective.
\end{lemma}
\begin{proof}
Suppose that $f \in {\rm{Ker}}(\Phi_n)$.  Then
$$  \Phi_n(f)(g_0, \, g_1, \, \ldots \, , \, g_n) = 0  $$
for all $(g_0, \, g_1, \, \ldots \, , \, g_n) \in G^{n+1}$.  Let 
$f(g_1, \, \ldots \, , \, g_n) = \sum_{i=1}^m c_i h_i \in k[G]$, where
$c_i \in k$ and the $h_i$ are distinct elements of $G$.  For each
$h_j$, we have
\begin{align*}
\Phi_n(f)(h_j^{-1}, \, g_1, \, \ldots \, , \, g_n) &  = \sum_{i=1}^m c_i
\langle h_j^{-1}, \ h_i \rangle \\
& = c_j \langle h_j^{-1}, \ h_j \rangle = 0 .  
\end{align*}
Thus, $c_j = 0$, and $f: k[G]^{\otimes n} \to k[G]$  is the zero map.
\end{proof}

We adopt the following notation for elements of 
${\rm{Hom}}_k (k[G]^{\ot n}, \, k[G])$ and 
${\rm{Hom}}_k (k[G]^{\ot (n+1)}, \, k)$, recalling that $k[G]$ is a
free $k$-module with basis given by the elements of $G$.  For $g_0$,
$g_1$, $\ldots \,$, $g_n \in G$ and $h_1$, $h_2$, $\ldots \, $, $h_n
\in G$, let
$$  (g_0, \, g_1, \, \ldots \, , \, g_n)^{\#} : k[G]^{\ot n} \to k[G]  $$
denote the $k$-linear map determined by 
$$  (g_0, \, g_1, \, \ldots \, , \, g_n)^{\#} (h_1, \, h_2, \, \ldots
\, , \, h_n) = \begin{cases} g_0, & h_1 = g_1, \, \ldots \, , \, h_n =
                                                              g_n , \\
                             0,    & {\rm{otherwise}}.  \end{cases}  $$
Additionally, for $h_0 \in G$, let $(g_0, \, g_1, \, \ldots \, , \,
g_n)^{*}: k[G]^{\ot (n+1)} \to k$ be the $k$-linear map determined
by
$$  (g_0, \, g_1, \, \ldots \, , \, g_n)^{*} (h_0, \, h_1, \, \ldots
\, , \, h_n) = \begin{cases} 1, & 
h_0 = g_0, \,  h_1 = g_1, \, \ldots \, , \, h_n = g_n, \\
                             0,  & {\rm{otherwise}}.  \end{cases}  $$
Under this notation,
$$  \Phi_n \big( (g_0, \, g_1, \, g_2, \, \ldots \, , \, g_n)^{\#} \big) =
(g^{-1}_0, \, g_1, \, g_2, \, \ldots \, , \, g_n)^* .  $$

Let $I_n = {\rm{Im}} \, \Phi_n \subseteq {\rm{Hom}}_k(k[G]^{\ot (n+1)},
\, k)$.  Then $I_* = \{ I_n \}_{n \geq 0}$ is a subcomplex of 
${\rm{Hom}}_k(k[G]^{\ot (*+1)}, \, k)$.  Note that for fixed
$(g_1, \, g_2, \, \ldots \, , \, g_n) \in G^n$, an element $\al \in I_n$
can be written as a finite sum
$$  \al ( \underline{\ \ }\, , \, g_1, \, g_2 , \, \ldots \, , \, g_n) = 
\sum_{i=1}^m c_i(h_i, \, g_1, \, g_2 , \, \ldots \, , \, g_n)^*.  $$

\begin{lemma}
There is a cochain map (in fact, a cochain isomorphism) 
$$  \Psi_n : I_n \to {\rm{Hom}}_k (k[G]^{\ot n}, \, k[G]), \ \ \  n
\geq 0,  $$
induced by 
$$  \Psi_n \big( (g_0, \, g_1, \, g_2, \, \ldots \, , \, g_n)^* \big) =
(g^{-1}_0, \, g_1, \, g_2, \, \ldots , \, , g_n)^{\#} .  $$  
\end{lemma}
\begin{proof}
First, $\Psi_n$ can be extended linearly over finite sums
$$  \Psi_n \big( \al ( \underline{\ \ }\, , \, g_1, \, g_2 , \, \ldots
\, , \, g_n) \big) = \sum_{i=1}^m c_i 
(h^{-1}_i, \, g_1, \, g_2 , \, \ldots \, , \, g_n)^{\#}.  $$
Second, since $\Phi_n$ is injective and a cochain map, it follows that
the following diagram commutes:
\begin{equation}
\begin{CD}
I_n  @>b^{*}>>  I_{n+1}        \\
@V{\Psi_n}VV  @VV{\Psi_{n+1}}V \\
{\rm{Hom}}_k (k[G]^{\ot n}, \, k[G]) @>>{\delta}> 
{\rm{Hom}}_k (k[G]^{\ot (n+1)}, \, k[G]) . 
\end{CD} 
\end{equation}
\end{proof}

Let $G$ be a discrete group and let $B_*(G)$ be the simplicial bar
construction on $G$.  By definition, $B_n(G) = G^n$, $n = 0$, 1, 2,
$\ldots \, $, with face maps 
\begin{align*}
& d_i : G^n \to G^{n-1}, \ \ \ i = 0, \ 1, \ 2, \, \ldots \, , \, n, \\
& d_i (g_1, \, g_2, \, \ldots \, , \, g_n) = \begin{cases} (g_2, \, \ldots \,
               , \, g_n), & i = 0, \\
               (g_1, \, g_2, \, \ldots \, , \, g_i g_{i+1}, \, \ldots
               \, , \, g_n), & i = 1, \, 2, \, \ldots \, , \, n-1, \\
               (g_1, \, g_2, \, \ldots \, , \, g_{n-1}),  & i =
               n, \end{cases}
\end{align*}
and degeneracies 
\begin{align*}
& s_i : G^n \to G^{n+1}, \ \ \ i= 0, \ 1, \ \ldots \, , \ n, \\
& s_i(g_1, \, g_2, \, \ldots \, , \, g_n) = (g_1, \, \ldots \, , \,
g_i, \, 1, \, g_{i+1}, \, \ldots \, , \, g_n).  
\end{align*}
Of course, the geometric realization $|B_*(G)|$ is a model for the
classifying space $BG$, up to homotopy.
Let $N^{\rm{cy}}_*(G)$ denote the cyclic bar construction
\cite[7.3.10]{cyclic-hom} on $G$ with $N^{\rm{cy}}_n(G) = G^{n+1}$, $n =
0$, 1, 2, $\ldots \ $.  The face maps $d_i : G^{n+1} \to G^n$ and
degeneracies $s_i : G^{n+1} \to G^{n+2}$ are given by adopting the
formulas  \eqref{face}--\eqref{degeneracies}.  For the geometric
realization $|N^{\rm{cy}}_*(G)|$, there is a homotopy equivalence 
\cite{BF, Goodwillie} \cite[7.3.11]{cyclic-hom} 
$$  \lambda : |N^{\rm{cy}}_*(G)| \overset{\simeq}{\longrightarrow} 
{\rm{Maps}}(S^1, \ BG) := {\cal{L}}BG,  $$
where $S^1$ denotes the unit circle.  Let $H_*$ denote singular
homology and $H^*$ denote singular cohomology.  There are isomorphisms
\begin{align*}
&  \lambda_* : HH_*(k[G]; \, k[G]) \overset{\simeq}{\longrightarrow}
H_* ( {\cal{L}}BG; \, k), \\
& \lambda^* : H^* ( {\cal{L}}BG; \, k) \overset{\simeq}{\longrightarrow}
HH^*_{\cal{K}}( k[G] ).
\end{align*}
The map $\lambda^*$ preserves the cup-$i$ products.  For 
$\al \in H^p ( {\cal{L}}BG; \, k)$ and 
$\be \in H^q ( {\cal{L}}BG; \, k)$, we have
$$  \lambda^* (\al \underset{i, \, S}{\cdot} \be ) =
\lambda^* (\al) \underset{i, \, S}{\cdot} \lambda^*(\be).  $$
Thus, with $\bz /2$ coefficients, $\lambda^*$ is a map of modules over the
Steenrod algebra.

There are maps of simplicial sets
\begin{align*}
&  \iota : B_*(G) \to N^{\rm{cy}}_*(G), \ \ \  \pi : N^{\rm{cy}}_*(G) \to
B_*(G), \\
&  \iota : G^n \to G^{n+1}, \ \ \ \pi: G^{n+1} \to G^n , \\
&  \iota (g_1, \, g_2, \, \ldots \, , \, g_n) = ( (g_1 g_2 \ldots
g_n)^{-1}, \, g_1, \, g_2, \, \ldots \, , \, g_n), \\
&  \pi (g_0, \, g_1, \, g_2, \, \ldots \, , \, g_n) = 
(g_1, \, g_2, \, \ldots \, , \, g_n).    
\end{align*}
Let $\iota_* : H_*(BG; \, k) \to H_*({\cal{L}}BG; \, k)$ and 
$\pi_* : H_*({\cal{L}}BG; \, k) \to H_*(BG; \, k)$ be the induced maps 
on homology, $\iota^*  : H^*({\cal{L}}BG; \, k) \to H^*(BG; \, k)$,
$\pi^* : H^*(BG; \, k) \to H^*({\cal{L}}BG; \, k)$ the induced maps on
cohomology.  Since $ \pi \circ \iota = {\bf{1}}$ on $B_*(G)$, we have
splittings of $k$-modules:
\begin{align*}
&  H_*({\cal{L}}BG; \, k) \simeq H_*(BG; \, k) \oplus
{\rm{Ker}}(\pi_*) \\
&  H^*({\cal{L}}BG; \, k) \simeq H^*(BG; \, k) \oplus
{\rm{Ker}}(\iota^*).  
\end{align*}

\begin{lemma}
Let $N^{\rm{cy}}_n(G, \, e) = \{ (g_0, \, g_1, \, \ldots \, , \, g_n)
\in G^{n+1} \, | \, g_0g_1 \ldots g_n =e \}$ for $n = 0, \ 1, \ 2, \, 
\ldots \ $.  Then
\begin{itemize}
\item[(i)] $N^{\rm{cy}}_*(G, \, e)$ is a subsimplicial set of 
$N^{\rm{cy}}_*(G)$.
\item[(ii)] Im$(\iota ) = N^{\rm{cy}}_*(G, \, e)$.
\end{itemize}
\end{lemma}
\begin{proof}
Part (i) follows since $N^{\rm{cy}}_*(G, \, e)$ is closed under the
face maps and degeneracies of $N^{\rm{cy}}_*(G)$.  For part (ii), let 
$ (g_0, \, g_1, \, \ldots \, , \, g_n) \in N^{\rm{cy}}_n(G, \, e). $
Then $g_0 g_1 \ldots g_n = e$ and $g_0 = (g_1 g_2 \ldots g_n)^{-1}$.
Thus,  
$$  \iota (g_1, \, g_2, \, \ldots \, , g_n) = (g_0 , \, g_1, \,
g_2, \, \ldots \, , g_n).  $$  
\end{proof} 
\begin{definition} \label{supported}
For fixed $\al_0$, $\al_1$, $\ldots \, $, $\al_p \in G$ with 
$$  \al = (\al_0, \, \al_1, \, \ldots \, , \, \al_p)^* \in
{\rm{Hom}}_k (k[G]^{\ot (p+1)}, \, k),  $$
we say that $\al$ is supported on $BG \simeq |N^{\rm{cy}}_*(G, \, e)|$
if the product 
$$ \al_0 \al_1 \ldots \al_p = e.  $$  
\end{definition}
Recall that $k[G^{n+1}] \simeq k[G]^{\ot (n+1)}$ is a free $k$-module
with basis given by the elements of $G^{n+1}$.  Thus, 
${\rm{Hom}}_k ( k[ N^{\rm{cy}}_*(G, \, e)], \ k)$ is a submodule of
${\rm{Hom}}_k ( k[G]^{\ot (*+1)}, \ k)$.   An element $\gamma \in
{\rm{Hom}}_k ( k[ N^{\rm{cy}}_n(G, \, e)], \ k)$ is extended to 
${\rm{Hom}}_k ( k[G]^{\ot (n+1)}, \ k)$ by setting
$\gamma (g_0, \, g_1, \, \ldots \, , \, g_n) = 0$ for $g_0 g_1 \,
\ldots \, g_n \neq e$.  A direct calculation of the coboundary $b^*$
shows that ${\rm{Hom}}_k ( k[ N^{\rm{cy}}_*(G, \, e)], \ k)$ is a
subcomplex of ${\rm{Hom}}_k ( k[G]^{\ot (*+1)}, \ k)$.  For $\al$
supported on $BG$ as defined above, $\al \in 
{\rm{Hom}}_k ( k[ N^{\rm{cy}}_*(G, \, e)], \ k)$.

\section{Cup-$i$ Products in Hochschild Cohomology}

In this section we transport the cup-$i$ products to the cochain
complex  ${\rm{Hom}}_k (k[G]^{\ot *}, \, k[G])$ by showing that the
complex $I_* = {\rm{Im}} \, \Phi_*$ is closed under Steenrod's cup-$i$
products, $i \geq 0$.   For $i = 0$, 
$$  \al \underset{0, \, S}{\cdot} \be = \al \underset{S}{\cdot}\be  $$  
is the simplicial cup product.  The Gerstenhaber and the simplicial cup
products define two product structures on $HH^*(k[G]; \, k[G])$ that
agree as cochains when evaluated on a subcomplex that represents $BG
\simeq |N^{\rm{cy}}_*(G, \, e)|$.  Gerstenhaber's pre-Lie
product agrees with Steenrod's cup-one product on this subcomplex as well.  
Unless otherwise stated, $k$ denotes a unital, commutative coefficient
ring, such as the integers.  In the special case of $\bz /2$
coefficients, the cup-$i$ products induce an action of the mod 2
Steenrod algebra on $HH^* ( \bz /2 [G]; \, \bz /2 [G])$, which
recovers the action of the Steenrod algebra on the free loop space
${\cal{L}}BG$.  

\begin{lemma}  \label{cup-closed} The subcomplex 
$I_* = {\rm{Im}}\, \Phi_* \subseteq {\rm{Hom}}_k(k[G]^{\ot {* + 1}}, \, k)$ 
is closed under the simplicial cup product.
\end{lemma}
\begin{proof}
Let $\al \in I_p$ and $\be \in I_q$.  Given any $(\al_1, \, \al_2, \,
\ldots \, , \, \al_p) \in G^p$, there are only finitely many $h_i \in
G$ with
$$  \al = \sum_{i=1}^m c_i(h_i, \, \al_1, \, \al_2, \, \ldots \, , \,
\al_p)^*, \ \ \ c_i \in k, \ \ c_i \neq 0.  $$
By linearity, consider the case $i = 1$, and for ease of notation,
let
$$  \al = (\al_0, \, \al_1, \, \al_2, \, \ldots \, , \, \al_p)^* .  $$
Similarly, consider
$$  \be = (\be_0, \, \be_1, \, \be_2, \, \ldots \, , \, \be_q)^* .  $$
For $g_i \in G$, 
\begin{align*}
&  (\al \underset{S}{\cdot} \be) (g_0, \, g_1, \, \ldots \, , \, g_p,
\, g_{p+1}, \, \ldots \, , \, g_{p+q}) = \\
&  \al ( (g_{p+1}g_{p+2} \ldots g_{p+q}g_0), \ g_1, \ g_2, \ \ldots \,
, \ g_p) \be ( (g_0g_1 \ldots g_p), \ g_{p+1}, \ \ldots \, , \ g_{p+q}). 
\end{align*}
Thus, $\al \underset{S}{\cdot} \be$ is non-zero if and only if
\begin{align*}
&  (g_{p+1}g_{p+2} \ldots g_{p+q}g_0) = \al_0, \ \ g_1 = \al_1, \ \ g_2 =
\al_2, \  \ldots \, ,  \ g_p = \al_p , \\
&   (g_0g_1 \ldots g_{p-1} g_{p}) = \be_0, \ \ g_{p+1} = \be_1,\  \ g_{p+2}
= \be_2, \  \ldots \, ,  \ g_{p+q} = \be_q .  
\end{align*} 
This system of equations is over-determined and necessary conditions
that $\al \underset{S}{\cdot} \be \neq 0$ are
$$  g_0 = (\be_1 \be_2 \, \ldots \, \be_q)^{-1} \al_0 \ \ \ {\rm{and}}
\ \ \ g_0 = \be_0 (\al_1 \al_2 \, \ldots \, \al_p)^{-1}.  $$
Thus, $\al_0 \al_1 \, \ldots \, \al_p = \be_1 \be_2 \, \ldots \, \be_q
\be_0$ in order that $\al \underset{S}{\cdot} \be \neq 0$, in which case
$$   \al \underset{S}{\cdot} \be = (\be_0 (\al_1 \ldots \al_p)^{-1}, \
\al_1, \ \al_2, \ \ldots \, , \ \al_p, \  \be_1, \ \be_2, \ 
\ldots \, , \ \be_q)^{*}.  $$
Hence, there is at most only one possible choice for $g_0$ that yields
a non-zero result for 
$$  (\al \underset{S}{\cdot} \be)(\underline{\ \ },\, g_1, \, g_2, \,
\ldots \, , \, g_{p+q}).  $$  
It follows that $\al \underset{S}{\cdot} \be \in I_{p+q} = {\rm{Im}}\,
\Phi_{p+q}$.  
\end{proof}
\begin{corollary}  Let 
$$ \al = (\al_0, \, \al_1, \, \al_2, \, \ldots \, , \, \al_p)^* \in
I_p, \ \ \  \be = (\be_0, \, \be_1, \, \be_2, \, \ldots \, , \,
\be_q)^* \in I_q.  $$  
If the elements given by the products
$\al_0 \al_1 \, \ldots \, \al_p$ and $\be_0 \be_1 \, \ldots \, \be_q$
are in different conjugacy classes of $G$, then 
$\al \underset{S}{\cdot} \be = 0$.  
\end{corollary}

\begin{definition}
For $f \in {\rm{Hom}}_k(k[G]^{\ot p}, \, k[G])$ and
$g \in {\rm{Hom}}_k(k[G]^{\ot q}, \, k[G])$, define the simplicial cup
product 
$$  f \underset{S}{\cdot} g \in  {\rm{Hom}}_k(k[G]^{\ot {(p+q)}}, \,
k[G])  $$
by  $f \underset{S}{\cdot} g = \Psi_{p+q}( \Phi_p(f)
\underset{S}{\cdot} \Phi_q (g))$.  
\end{definition}
\begin{lemma}
For $f \in HH^p(k[G]; \, k[G])$ and $g \in HH^q(k[G]; \, k[G])$, we
have $f \underset{S}{\cdot} g \in   HH^{p+q}(k[G]; \, k[G])$.  
\end{lemma}
\begin{proof}
That the simplicial cup product is well-defined on $HH^{*}(k[G]; \,
k[G])$ follows from
\begin{align*}
\delta ( f \underset{S}{\cdot} g) & = \delta (\Psi_{p+q}( \Phi_p(f)
\underset{S}{\cdot} \Phi_q (g))) \\
& = \Psi_{p+q+1} ( b^*  ( \Phi_p(f) \underset{S}{\cdot} \Phi_q (g)) ) \\
& = \Psi_{p+q+1}(b^*(\Phi_p(f)) \underset{S}{\cdot} \Phi_q(g) +  
(-1)^p \Phi_p (f) \underset{S}{\cdot} b^*( \Phi_q (g)) ) \\
& = \Psi_{p+q+1}( \Phi_{p+1} (\delta f) \underset{S}{\cdot} \Phi_q (g) +
(-1)^p \Phi_p (f) \underset{S}{\cdot} \Phi_{q+1} (\delta g) )  \\
& = (\delta f) \underset{S}{\cdot} g + (-1)^p f  \underset{S}{\cdot}
(\delta g).
\end{align*}
\end{proof}

\begin{lemma}  \label{cup-one-closed} 
The subcomplex $I_* = {\rm{Im}} \, \Phi_*$ is closed under the 
cup-one product.
\end{lemma}  
\begin{proof}
Let $\al \in I_p$, $\be \in I_q$ and consider the case
$$  \al = (\al_0 , \, \al_1, \, \ldots \, , \, \al_p)^*, \ \ \ 
\be = (\be_0 , \, \be_1, \, \, \ldots \, , \, \be_q )^*  $$
as in Lemma \eqref{cup-closed}.  Let $\sigma = (g_0, \, g_1, \, \ldots
\, , \, g_{p+q-1}) \in G^{p+q}$.  The cup-one product $\al
\underset{1, \, S}{\cdot} \be$ is a sum of $p$-many terms indexed by
$j = 0$, 1, 2, $\ldots \,$, $p-1$.  Consider the $j$th term
\begin{align*}
& (\al \underset{1, \, S}{\cdot} \be)_j (\sigma ) = \\
& \al (g_0, \ g_1, \ \ldots \, , \ g_j, \ (g_{j+1}g_{j+2} \ldots
g_{j+q}), \ g_{j+q+1}, \ \ldots \, , \ g_{p+q-1}) \\
& \cdot \be ( (g_{j+q+1}g_{j+q+2} \ldots g_{p+q-1}g_0 g_1 \ldots g_j),
\ g_{j+1}, \ g_{j+2}, \ \ldots \, , \ g_{j+q})
\end{align*}  
In order that $(\al \underset{1, \, S}{\cdot} \be)_j
(\sigma ) \neq 0$, we need 
\begin{align*}
& g_0 = \al_0, \ \ g_1 = \al_1, \ \ldots \, , \ g_j = \al_j, \ \ 
(g_{j+1} g_{j+2} \ldots g_{j+q}) = \al_{j+1}, \\
& g_{j+q+1} = \al_{j+2}, \ \ g_{j+q+2} = \al_{j+3}, 
\ \ldots \, , \ g_{p+q-1} = \al_p , \\
& (g_{j+q+1}g_{j+q+2} \ldots g_{p+q-1}g_0g_1 \ldots g_j) = \be_0, \\
& g_{j+1} = \be_1, \ \ g_{j+2} = \be_2, 
\ \ldots \, , \ g_{j+q} = \be_q . 
\end{align*} 
The above system of equations is over-determined and necessary
conditions that 
$(\al \underset{1, \, S}{\cdot} \be)_j  \neq 0$ 
are 
$$  \be_1 \be_2 \, \ldots \, \be_q = \al_{j+1} \ \ \ 
{\rm{and}} \ \ \ \al_{j+2} \al_{j+3} \, \ldots \, \al_p \al_0 \al_1 \,
\ldots \, \al_j = \be_0,  $$  
in which case
$$  (\al \underset{1, \, S}{\cdot} \be)_j = (\al_0, \ \al_1, \ \ldots
\, , \ \al_j, \ \be_1, \ \be_2, \ \ldots \, , \ \be_q, \ \al_{j+2}, \
\al_{j+3}, \ \ldots \, , \ \al_p)^* .  $$
It follows that $\al \underset{1, \, S}{\cdot} \be \in I_{p+q-1}$.  
\end{proof}
\begin{lemma}
The subcomplex $I_*$ is closed under the cup-$i$ products, $i \geq 2$.
\end{lemma}
\begin{proof}
Let $\al \in I_p$, $\be \in I_q$.  
The cup-$i$ products are given in terms of overlapping
partitions of $(g_0, \, g_1, \, \ldots \, , \, g_{p+q-i})$ with $i+2$
many pieces \cite{McClure}.  Each overlapping partition yields a
summand of $\al \underset{i, \, S}{\cdot} \be$, and 
each summand of 
$$  (\al \underset{i, \, S}{\cdot} \be) (g_0, \, g_1, \, \ldots \, , \, g_{p+q-i}) $$ 
is over-determined by $(i+1)$-many equations.  Thus, given
$(g_1, \, g_2, \, \ldots \, , \, g_{p+q-i}) \in G^{p+q-i}$, for each
summand $(\al \underset{i, \, S}{\cdot} \be)_{(j_1, \, j_2, \, \ldots
  \, , j_i)}$, there is only one
possible choice of $g_0$ with
$$   (\al \underset{i, \, S}{\cdot} \be)_{(j_1, \, j_2, \, \ldots
  \, , j_i)}(g_0, \, g_1, \, \ldots \, , \, g_{p+q-i}) \neq 0 .  $$
Thus, $\al \underset{i, \, S}{\cdot} \be \in I_{p+q-i}$.  
\end{proof}

\begin{definition}
For $f \in {\rm{Hom}}_k(k[G]^{\ot p}, \, k[G])$ and
$g \in {\rm{Hom}}_k(k[G]^{\ot q}, \, k[G])$, define the cup-$i$ 
product 
$$  f \underset{i, \, S}{\cdot} g \in  {\rm{Hom}}_k(k[G]^{\ot {(p+q-i)}}, \,
k[G])  $$
by  $f \underset{i, \, S}{\cdot} g = \Psi_{p+q-i}( \Phi_p(f)
\underset{i, \, S}{\cdot} \Phi_q (g))$.  
\end{definition}
\begin{lemma}
For $i \geq 1$, 
\begin{align*}
& \delta( f \underset{i, \, S}{\cdot} g) = (\delta f ) \underset{i, \, S}{\cdot} g 
+ (-1)^{p-1} f  \underset{i, \, S}{\cdot} (\delta g) \\
& \hspace{.7in} + (-1)^p [ (-1)^{(i-1)(p+q+1)} f  \underset{i-1, \, S}{\cdot} g - 
(-1)^{pq} g  \underset{i-1, \, S}{\cdot} f ].
\end{align*}
\end{lemma}
\begin{proof}
We have:
\begin{align*}
\delta( f \underset{i, \, S}{\cdot} g) & = \delta ( \Psi_{p+q-i}( \Phi_p(f)
\underset{i, \, S}{\cdot} \Phi_q (g))) \\
& = \Psi_{p+q+1-i} b^* (\Phi_p(f) \underset{i, \, S}{\cdot} \Phi_q (g))\\
& = \Psi_{p+q+1-i} [ b^*(\Phi_p(f)) \underset{i, \, S}{\cdot} \Phi_q (g)
+ (-1)^{p-1} \Phi_p (f) \underset{i, \, S}{\cdot} b^*( \Phi_q(g))] \\
&\  + (-1)^p \Psi_{p+q+1-i}[ (-1)^{(i-1)(p+q+1)} \Phi_p (f) \underset{i-1, \, S}{\cdot}
\Phi_q(g) \\ 
& - (-1)^{pq} \Phi_q(g) \underset{i-1, \, S}{\cdot} \Phi_p(f)]\\
& = (\delta f ) \underset{i, \, S}{\cdot} g 
+ (-1)^{p-1} f  \underset{i, \, S}{\cdot} (\delta g) \\
& \  + (-1)^p [ (-1)^{(i-1)(p+q+1)} f  \underset{i-1, \, S}{\cdot} g - 
(-1)^{pq} g  \underset{i-1, \, S}{\cdot} f ].
\end{align*}
\end{proof} 

\begin{corollary}
There is a natural action of the mod 2 Steenrod algebra on $HH^* ( \bz
/2 [G]; \, \bz /2 [G])$, 
$$  Sq^i :  HH^p (\bz /2 [G]; \, \bz /2 [G]) \rightarrow
 HH^{p+i} (\bz /2 [G]; \, \bz /2 [G]),  $$
given by
$$  Sq^i (f) := Sq_{p-i} (f) = f \underset{p-i, \, S}{\cdot} f . $$
\end{corollary}
\begin{proof}
If $f : ( \bz /2 [G])^{\ot p} \to \bz /2 [G]$ is a cocycle mod 2, then
it follows that $ f \underset{p-i, \, S}{\cdot} f$ is also.  For more details 
see, for example, Mosher and Tangora \cite[pp. 16--18]{Mosher}.  
\end{proof}
Thus, the action of the Steenrod algebra on group cohomology, i.e.,
the mod 2 cohomology of $BG$, can be seen in Hochschild's original
complex defining $HH^*$.

We now show that for cochains supported on $BG$, the simplicial cup
product agrees with Gerstenhaber's product.  Also, Steenrod's cup-one
product agrees with Gerstenhaber's pre-Lie product for these
cochains.  Recall Definition \eqref{supported}.   
\begin{lemma} \label{Psi}
If 
\begin{align*}
&  \al = (\al_0, \, \al_1, \, \ldots \, , \, \al_p)^* \in
{\rm{Hom}}_k (k[G]^{\ot (p+1)}, \, k) \ \ \ {\rm{and}} \\
&  \be = (\be_0, \, \be_1, \, \ldots \, , \, \be_q)^* \in
{\rm{Hom}}_k (k[G]^{\ot (q+1)}, \, k)
\end{align*}
are supported on $BG$, then as cochains
$$  \Psi_{p+q}(\al \underset{S}{\cdot} \be) = \Psi_p(\al)
\underset{G}{\cdot} \Psi_q (\be) .  $$ 
In other words, for $f = \Psi_p(\al)$ and $g = \Psi_q(\be)$, we have
$$  f \underset{S}{\cdot} g = f \underset{G}{\cdot} g.  $$
\end{lemma} 
\begin{proof}
Let $g_i \in G$ for $i = 0$, 1, 2, $\ldots \,$, $p+q$, and let $\sigma = (g_0, \,
g_1, \, \ldots \, , \, g_{p+q})$.  Then 
$$  (\al \underset{S}{\cdot} \be ) \in 
{\rm{Hom}}_k (k[G]^{\ot (p+q+1)}, \, k)  $$ 
is determined by $(\al \underset{S}{\cdot} \be)(\sigma)$.  Necessary
conditions for $(\al \underset{S}{\cdot} \be)(\sigma) \neq 0$ are
stated in Lemma \eqref{cup-closed}.  Under the assumption that 
$\al_0 \al_1 \, \ldots \, \al_p = e$ and 
$\be_0 \be_1 \, \ldots\, \be_q = e$, we have 
$\al_0 \al_1 \, \ldots \, \al_p = \be_1 \be_2 \, \ldots \,\be_q \be_0$.  
For $(\al \underset{S}{\cdot} \be)(\sigma) \neq 0$, we need
$$  g_0 = (\be_1 \be_2 \ldots \be_q)^{-1} \al_0 = \be_0 \al_0 . $$
Thus,
\begin{align*}
&  \Psi_{p+q} (\al \underset{S}{\cdot} \be) \\
&  = \Psi_{p+q} \big( (\be_0 \al_0, \ \al_1, \ \al_2, \ \ldots \, ,
\al_p, \ \be_1, \ \be_2, \ \ldots \, , \ \be_q)^* \big) \\
& = ( (\be_0 \al_0)^{-1}, \ \al_1, \ \al_2, \ \ldots \, ,
\al_p, \ \be_1, \ \be_2, \ \ldots \, , \ \be_q)^{\#} \\
& = (\al_0^{-1} \be_0^{-1}, \ \al_1, \ \al_2, \ \ldots \, ,
\al_p, \ \be_1, \ \be_2, \ \ldots \, , \ \be_q)^{\#} \\
& = (\al_0^{-1}, \ \al_1, \ \ldots \, , \ \al_p)^{\#}
\underset{G}{\cdot}  (\be_0^{-1}, \ \be_1, \ \ldots \, , \ \be_q)^{\#} \\
& = \Psi_p(\al) \underset{G}{\cdot} \Psi_q(\be).  
\end{align*}
It follows that $f \underset{S}{\cdot} g = f \underset{G}{\cdot}g$
for $f = \Psi_p(\al)$, $g = \Psi_q(\be)$.  
\end{proof}

Recall that for $\al \in {\rm{Hom}}_k(k[G]^{\ot (p+1)}, \, k)$,  
$\be \in  {\rm{Hom}}_k(k[G]^{\ot (q+1)}, \, k)$, and $\sigma \in
G^{p+q}$, the $j$th term in Steenrod's cup-one product is given by 
\begin{align*}
& (\al \underset{1, \, S}{\cdot}  \be)_j (\sigma) = 
\al (d_{j+1}  d_{j+2}  \, \ldots \,  d_{j+q-1}) (\sigma)  \\
& \hspace{.5in} \cdot \be( d_0  d_1  \, \ldots \,  d_{j-1} 
d_{j+q+1}  d_{j+q+2}  \, \ldots \,  d_{p+q-1}) (\sigma)
\end{align*}  
\begin{theorem} \label{j-th-term}
Let 
\begin{align*}
&  \al = (\al_0, \, \al_1, \, \ldots \, , \, \al_p)^* \in
{\rm{Hom}}_k (k[G]^{\ot (p+1)}, \, k) \ \ \ {\rm{and}} \\
&  \be = (\be_0, \, \be_1, \, \ldots \, , \, \be_q)^* \in
{\rm{Hom}}_k (k[G]^{\ot (q+1)}, \, k)
\end{align*}
be supported on $BG$.  Then as cochains
$$  \Psi_{p+q-1}( (\al \underset{1, \, S}{\cdot}  \be)_j  ) = 
\Psi_p(\al) \underset{(j)}{\circ} \Psi_q (\be).  $$
For $f = \Psi_p(\al)$ and $g = \Psi_q(\be)$, we have
$$  (f \underset{1, \, S}{\cdot} g)_j = 
f \underset{(j)}{\circ} g.  $$
\end{theorem}
\begin{proof}
Let $\sigma = (g_0, \, g_1, \, \ldots \, , g_p, \, g_{p+1}, \, \ldots
\, , g_{p+q-1}) \in
G^{p+q}$.  Necessary conditions for $( \al \underset{1, \, S}{\cdot} \be)_j
(\sigma)$ to be non-zero are stated in Lemma \eqref{cup-one-closed}.   
Under the assumption that $\al_0 \al_1 \, \ldots \, \al_p = e$ and 
$\be_0 \be_1 \, \ldots\, \be_q = e$, we have 
$$  (\be_1 \be_2 \, \ldots \, \be_q = \al_{j+1}) \Longleftrightarrow  
( \al_{j+2} \al_{j+3} \, \ldots \, \al_p \al_0 \al_1 \, \ldots \,
\al_j = \be_0 ).  $$
In order that $( \al \underset{1, \, S}{\cdot} \be)_j (\sigma) \neq
0$, we need $\be_1 \be_2 \, \ldots \, \be_q = \al_{j+1}$, in which
case $\al_{j+1} = \be_0^{-1}$ and 
$$  (\al \underset{1, \, s}{\cdot} \be)_j = (\al_0, \ \al_1, \ \ldots
\, , \ \al_j, \ \be_1, \ \be_2, \ \ldots \, , \ \be_q, \ \al_{j+2}, \
\al_{j+3}, \ \ldots \, , \ \al_p)^*.  $$
If $\al_{j+1} \neq \be_0^{-1}$, then $(\al \underset{1, \, S}{\cdot} \be)_j
= 0$.  

Now, 
\begin{align*}
& f = \Psi_p(\al ) = (\al_0^{-1}, \ \al_1, \ \ldots \, , \ \al_p)^{\#} \\
& g = \Psi_p(\be ) = (\be_0^{-1}, \ \be_1, \ \ldots \, , \ \be_q)^{\#}
\end{align*}
For $h_i \in G$, $i = 1$, 2, 3, $\ldots \,$, $p+q-1$, we have
\begin{align*} 
& (f \underset{(j)}{\circ} g)(h_1, \, h_2, \, \ldots \, , \,
h_{p+q-1}) \\
& = f(h_1, \, h_2, \, \ldots \, , \, h_j, \ g(h_{j+1}, \, \ldots \, ,
\, h_{j+q}), \ h_{j+q+1}, \, \ldots \, , \, h_{p+q-1}).
\end{align*}
It follows that 
$$  f \underset{(j)}{\circ} g = ( \al_0^{-1},\ \al_1, \ \ldots
\, , \ \al_j, \ \be_1, \ \be_2, \ \ldots \, , \ \be_q, \ \al_{j+2}, \
\al_{j+3}, \ \ldots \, , \ \al_p)^{\#}  $$        
under the condition that $\al_{j+1} = \be_0^{-1}$.  
Thus, 
\begin{align*}
&  \Psi_{p+q-1}( (\al \underset{1, \, S}{\cdot}  \be)_j  ) = 
\Psi_p(\al) \underset{(j)}{\circ} \Psi_q (\be) \ \ \ {\rm{and}}  \\
&  (f \underset{1, \, S}{\cdot} g)_j  =  f \underset{(j)}{\circ} g.  
\end{align*}
\end{proof}
\begin{corollary}  
Let 
\begin{align*}
&  \al = (\al_0, \, \al_1, \, \ldots \, , \, \al_p)^* \in
{\rm{Hom}}_k (k[G]^{\ot (p+1)}, \, k) \ \ \ {\rm{and}} \\
&  \be = (\be_0, \, \be_1, \, \ldots \, , \, \be_q)^* \in
{\rm{Hom}}_k (k[G]^{\ot (q+1)}, \, k)
\end{align*}
be supported on $BG$.  Then as cochains
$$  \Psi_{p+q-1}( \al \underset{1, \, S}{\cdot}  \be ) = 
\Psi_p(\al) {\circ} \Psi_q (\be),  $$
i.e., over $BG$ Steenrod's cup-one product
is Gerstenhaber's pre-Lie product, after application
of the cochain map $\Psi_*$.  For $f = \Psi_p (\al)$ and $g =
\Psi_q(\be)$, we have $f \underset{1, \, S}{\cdot} g = f \circ g$.    
\end{corollary}
\begin{proof}
The proof follows from Theorem \eqref{j-th-term}, the definition of
Steenrod's cup-one, and the definition of the pre-Lie product.
\end{proof}
\begin{corollary}
Let $\al = (\al_0, \, \al_1, \, \ldots \, , \, \al_p)^* \in
{\rm{Hom}}_{\bz /2} (\bz /2 [G]^{\ot (p+1)}, \, \bz /2)$ be a cocycle
supported on $BG$.  Set $f = \Psi_p (\al)$. 
Then on $HH^p ( \bz /2 [G]; \, \bz /2 [G])$, we have
$$  Sq^{p-1}(f) = f \underset{1, \, S}{\cdot} f = f \circ f.  $$
\end{corollary} 
\begin{corollary}  \label{zero-bracket}
Let 
\begin{align*} 
&  \al = (\al_0, \, \al_1, \, \ldots \, , \, \al_p)^* \in
{\rm{Hom}}_k (k[G]^{\ot (p+1)}, \, k) \ \ \ {\rm{and}} \\
&  \be = (\be_0, \, \be_1, \, \ldots \, , \, \be_q)^* \in
{\rm{Hom}}_k (k[G]^{\ot (q+1)}, \, k)
\end{align*}
be cocycles supported on $BG$.  Let $f = \Psi_p (\al)$ and $g =
\Psi_q(\be)$.  Then the Lie bracket
$$  [f, \, g] = f \circ g - (-1)^{(p+1)(q+1)} g \circ f  $$
is zero in $HH^{p+q-1}(k[G]; \, k[G])$.  
\end{corollary}
\begin{proof}
In ${\rm{Hom}}(k[G]^{(*+1)}, \, k)$, we have
$$  b^*( (-1)^{q+1} \al \underset{2, \, S}{\cdot} \be ) =
\al \underset{1, \, S}{\cdot} \be - (-1)^{(p+1)(q+1)} \be \underset{1,
  \, S}{\cdot} \al, $$
since $b^*(\al) = 0$ and $b^*(\be) = 0$.  Thus,
$$  \Psi_{p+q-1} ( b^*( (-1)^{q+1} \al \underset{2, \, S}{\cdot} \be )) =
\Psi_{p+q-1} ( \al \underset{1, \, S}{\cdot} \be   - (-1)^{(p+1)(q+1)} 
\be \underset{1, \, S}{\cdot} \al ),  $$
and  $[f, \, g]$ is a coboundary.  
\end{proof}

From the work of Tradler \cite{Tradler} and others \cite[13.7.6]{LV},
the Hochschild cohomology groups $HH^*( k[G]; \, k[G])$ support the
structure of a Batalin-Vilkovisky algebra with the Gerstenhaber
product and Lie bracket given as in this paper.  Additionally 
there is \cite{Tradler} a square-zero operator $\Delta$ on $HH^*
(k[G]; \, k[G])$ of degree $-1$.  
\begin{corollary}
Let $f \in HH^p( k[G]; \, k[G])$ and $g \in HH^q(k[G]; \, k[G])$ be
supported on $BG$.  Then in $HH^*(k[G]; \, k[G])$,  
$$  \Delta ( f \cdot g) = \Delta (f) \cdot g +
(-1)^p f \cdot \Delta (g), $$
where the product $f \cdot g$ can be taken to be either the simplicial 
cup product or the Gerstenhaber product.  
\end{corollary}
\begin{proof}
It follows from Corollary \eqref{zero-bracket} and \cite{Tradler} that
$$  0 = [f, \, g] = - (-1)^{(p-1)q} \big( \Delta (f \cdot g) - \Delta (f)
\cdot g - (-1)^p f \cdot \Delta (g) \big)  $$
in $HH^*(k[G]; \, k[G])$.    
\end{proof} 

In future work, we plan to extend these results to simplicial group
rings, in particular to a simplicial group whose geometric realization
is homotopy equivalent to $\Omega M$, the based loop space on a
manifold $M$ that is not necessarily simply connected.

\end{document}